\documentclass{amsart}

\usepackage{amsmath,amssymb,amsthm,comment}
\usepackage{amsrefs}
\usepackage{enumitem}
\usepackage[foot]{amsaddr}

\theoremstyle{plain}

\newtheorem{thm}{Theorem}[section]
\newtheorem{prop}[thm]{Proposition}

\theoremstyle{definition}

\newtheorem{remark}[thm]{Remark}

\newtheorem{definition}[thm]{Definition}

\newcommand{\p}{\mathfrak p}
\newcommand{\Q}{\mathbf Q}
\newcommand{\Ql}{\Q_\ell}
\newcommand{\C}{\mathbf C}
\newcommand{\Z}{\mathbf Z}
\newcommand{\Zhat}{\hat{\Z}}

\newcommand{\Fpbar}{\overline{\mathbf F}_p}
\newcommand{\Qpbar}{\overline{\Q}_p}
\newcommand{\Qp}{\Q_p}
\newcommand{\Fp}{{\mathbf F}_p}
\newcommand{\Qbar}{\overline{\Q}}

\renewcommand{\O}{\mathcal O}
\newcommand{\MF}{\mathcal{MF}}
\newcommand{\WFL}{\mathcal{W}_{\text{FL}}}
\newcommand{\Rep}{\mathrm{Rep}}
\newcommand{\A}{\mathbf A}
\newcommand{\R}{\mathbf R}

\newcommand{\Gl}{G_\ell}
\newcommand{\Gp}{G_p}

\newcommand{\GQSp}{G_{\Q,S \cup \{p\}}}
\newcommand{\Il}{I_\ell}

\newcommand{\Hl}{H^2_{\text{loc}}}
\newcommand{\Hg}{H^2_{\text{glo}}}
\newcommand{\rhob}{\bar{\rho}}

\newcommand{\kp}{k_\p}
\newcommand{\Kp}{\kp}
\newcommand{\Kpp}{K_{\pi,\p}}
\newcommand{\rpp}{\rho_{\pi,\p}}
\newcommand{\rbpp}{\rhob_{\pi,\p}}
\newcommand{\rppl}{\rho_{\pi,\p}|_{G_\ell}}
\newcommand{\rbppl}{\rhob_{\pi,\p}|_{G_\ell}}
\newcommand{\rppls}{\rho_{\pi,\p}|_{G_\ell}^{\ssi}}
\newcommand{\rbppls}{\rhob_{\pi,\p}|_{G_\ell}^{\ssi}}

\newcommand{\rss}{\rho_{\psi,\p}}
\newcommand{\rbss}{\rhob_{\psi,\p}}
\renewcommand{\wp}{\omega_{\psi}}
\newcommand{\wpp}{\omega_{\psi,\p}}

\newcommand{\cop}{\chi_{1,\p}}
\newcommand{\ctp}{\chi_{2,\p}}
\newcommand{\cp}{\chi_\p}
\newcommand{\cbp}{\bar{\chi}_\p}
\newcommand{\cycp}{\nu_\p}
\newcommand{\cycbp}{\bar{\nu}_\p}
\newcommand{\tp}{t_{\p}}

\renewcommand{\sp}{\sigma_\p}

\newcommand{\Ca}{\mathcal C}
\newcommand{\Sets}{\text{Sets}}
\newcommand{\univ}{\text{univ}}
\newcommand{\Ru}{R^{\univ}}
\newcommand{\rhou}{\rho^{\univ}}

\DeclareMathOperator{\ab}{ab}
\DeclareMathOperator{\ad}{ad}
\DeclareMathOperator{\diag}{diag}
\DeclareMathOperator{\Ext}{Ext}
\DeclareMathOperator{\Gal}{Gal}
\DeclareMathOperator{\GSp}{GSp}
\DeclareMathOperator{\ssi}{ss}
\DeclareMathOperator{\St}{St}
\DeclareMathOperator{\GL}{GL}
\DeclareMathOperator{\Hom}{Hom}
\DeclareMathOperator{\tr}{tr}
\DeclareMathOperator{\Ind}{Ind}
\DeclareMathOperator{\End}{End}
\DeclareMathOperator{\Frob}{Frob}

\newcommand{\gsp}{\mathfrak{gsp}}

\newcommand{\ar}{\ad \rpp}
\newcommand{\arb}{\ad \rbpp}

\newcommand{\arl}{\ad \rpp|_{G_\ell}}
\newcommand{\arbl}{\ad \rbpp|_{G_{\ell}}}

\newcommand{\artsl}{\arb(1)|_{G_{\ell}}^{\ssi}}
\newcommand{\trr}{{\bf 1}}

\newcommand{\inj}{\hookrightarrow}

\begin{document}

\title{Unobstructed deformation problems for $\GSp(4)$}
\author{Michael Broshi}
\email{michaelbroshi@gmail.com}
\address{The MathWorks, 1 Apple Hill Drive, Natick, MA 01760 }
\author{Mohammed Zuhair Mullath}
\email{mullath@math.umass.edu}
\address{Department of Mathematics and Statistics, University of Massachusetts Amherst, Amherst, MA 01003}
\author{Claus Sorensen}
\email{csorensen@ucsd.edu}
\address{Department of Mathematics, UC San Diego, La Jolla, CA 92093}
\author{Tom Weston}
\email{weston@math.umass.edu}
\address{Department of Mathematics and Statistics, University of Massachusetts Amherst, Amherst, MA 01003}

\date{\today}

\maketitle

\section{Introduction}

Let $\pi$ be a cuspidal automorphic representation for $\GSp(4)$ over $\Q$.  We assume that:
\begin{enumerate}
\item The archimedean component $\pi_{\infty}$ of $\pi$ is essentially discrete series with the same central
and infinitesimal character as the finite-dimensional irreducible algebraic representation of highest weight $a \geq b \geq 0$;
\item $\pi$ has a unitary twist of the form $\pi \otimes ||\cdot||^{w/2}$, with $w \in \Z$;
\item $\pi$ is globally generic (i.e., $\pi$ has a non-zero Whittaker-Fourier coefficient).
\end{enumerate}
Let $S$ denote the set of primes for which the local component of $\pi$ is ramified.  For a prime
$\p$ of the coefficient field of $\pi$ (a number field generated by the Satake parameters of the unramified components of $\pi$), let $\kp$ denote the corresponding residue field.  There is then
an associated Galois representation
$$\rbpp : \GQSp \to \GSp_4(\kp);$$
here $\GQSp$ denotes the Galois group of the maximal extension of $\Q$ unramified outside
$S$ and the residue characteristic $p$ of $\kp$.  When $\rbpp$ is absolutely irreducible, let $R_{\pi,\p}$ denote the associated universal
deformation ring: it is a quotient of a power series ring in $d_1$ variables over $W(\kp)$ by $d_2$ relations, where
$$d_i = \dim_{\kp} H^i(\GQSp,\ad \rbpp)$$
with $\ad \rbpp$ the $10$ dimensional adjoint representation of $\rbpp$.
Furthermore, $d_1-d_2$ is easily computable via an Euler characteristic calculation, equalling $5$ or $7$ depending on $\pi$.
In particular, when $d_2=0$, the ring $R_{\pi,\p}$ is simply isomorphic to a power series in $5$ or $7$ variables over $W(\kp)$; we then say that
the deformation theory of $\rbpp$ is {\it unobstructed}.

The Poitou--Tate exact sequence gives a natural way to decompose the obstruction space $H^2(\GQSp,\ad \rbpp)$ into a subspace
$\Hl(\GQSp,\ad \rbpp)$ of local obstructions and a quotient $\Hg(\GQSp,\ad \rbpp)$ of global obstructions.  (Note that no assumption on
irreducibility is required for these definitions.)  The goal of this paper is to prove the following.

\begin{thm}
Let $\pi$ be as above.  Assume also that $\pi$ is not supercuspidal at $2$.
If $a > b > 0$, then $\Hl(\GQSp,\ad \rbpp)=0$ for all but finitely many primes $\p$ such that $\rbpp(G_{\Q(\zeta_p)}) \subset \GL_4(\Fpbar)$ is an adequate subgroup.
\end{thm}

\begin{remark}
The large image (adequate) hypothesis is required for the level lowering approach in section 6 to work. It is expected that for $\pi$ which is neither CAP nor endoscopic, $\rbpp(G_{\Q(\zeta_p)})$ is adequate for all but finitely many $\rbpp$ in the compatible system. This is known to be true for primes $\p$ above a set of rational primes of Dirichlet density 1 by \cite[Theorem B]{Dinakar} and \cite[Proposition 5.3.2]{BLGGT}. We also remark that for $p \geq 11$, adequacy of $\rbpp(G_{\Q(\zeta_p)})$ is equivalent to absolute irreducibility of $\rbpp|_{G_{\Q(\zeta_p)}}$ (\textit{cf. }\cite[Theorem 9]{GHTT}).
\end{remark}

\begin{remark}
The assumption at $2$ could perhaps be removed with some additional group theory calculations.
\end{remark}

\begin{remark}
We expect the conclusion to hold even when $a=b$ or $b=0$, but are unable to prove it with our present methods.  It is perhaps not too difficult to prove it at least
for a set of primes of density one, but we have not pursued that result here.
\end{remark}

\begin{remark} It is not strictly necessary to assume that $\pi$ is globally generic. An analogous (but weaker) assumption on the global $L$-packet containing $\pi$ suffices. (See remark 3.2 for details.) In particular, Theorem 1.1 applies to all classical Holomorphic Siegel cusp forms of cohomological weights $k_1 \geq k_2 \geq 3$ which are not supercuspidal at 2. (The Harish-Chandra parameter of such a form is $(a, b) = (k_1-1, k_2-2)$.)
\end{remark}


Combined with the Taylor--Wiles results of \cite{GeT} it then in principle becomes possible to give examples of $\pi$ for which the deformation theory of
$\rbpp$ is unobstructed for all but finitely many $\p$.  We have not pursued such examples here.

 
We are indebted to Toby Gee for guidance on level lowering in Section 6 and to Wee Teck Gan and Kimball Martin for aid with Section 7.  We also thank
Matthew Emerton, Farshid Hajir, Jeff Hatley, Robert Pollack and Siman Wong for helpful conversations.

\section{Notation}

For a prime $\ell$, we write $\Gl$ for the absolute Galois group of $\Ql$.
If $\chi : \Ql^{\times} \to \C^{\times}$ is of arithmetic type, we write $\chi_\p : \Gl \to \Qpbar^{\times}$ for
the corresponding $p$-adic Galois character; here $\p$ is a prime of the image of $\chi$ containing $p$.  We write
$\nu : \Ql^{\times} \to \C^{\times}$ for the norm character; its associated character $\cycp : \Gl \to \Qp^{\times}$ is the
$\p$-adic cyclotomic character, with reduction $\cycbp : \Gl \to \Fp^{\times}$. If $V$ is a Galois
module, we write $V(1)$ for its cyclotomic Tate twist.

Let $\GSp_4$ denote the closed subscheme of $\GL_{4/\Z}$ of matrices preserving the symplectic form
$$J = \left[ \begin{array}{cccc} &&&1 \\ &&1& \\ &-1&& \\ -1&&& \end{array}\right].$$
We write $\gsp_4$ for the Lie algebra of $\GSp_4$. 

\section{Galois representations for $\GSp(4)$}

In this section we briefly review the known facts for Galois representations associated to automorphic representations for
$\GSp(4)$.   Let $\pi = \pi_{f} \otimes \pi_{\infty}$ be an irreducible cuspidal automorphic
representation of $\GSp_4(\A_{\Q})$ such that $\pi^\circ := \pi  \otimes || \cdot||^{w/2}$ is unitary for some $w \in \Z$. We assume that $\pi_{\infty}$ belongs to an (essentially) discrete series representation of $\GSp_4(\R)$ of (Harish-Chandra) weight $(a, b)$ for some integers  $a \geq b \geq 0$.  Associated to such a $\pi$
is a compatible system of Galois representations \cite[Theorem I]{Wei} constructed by Taylor, Laumon and Weissauer. 
More precisely, $\pi'_f = \pi_f^\circ \otimes || \cdot||^{-(a+b-3)/2}$ is defined over a number field $E$ \cite[\S 1]{Wei}:
$$ E := \{\alpha \in \C \mid \sigma(\alpha) = \alpha \; \forall
\sigma \in \textrm{Aut}(\C) \,\text{~for which~}\, \;(\pi'_f)^{\sigma} \simeq \pi'_f\}.$$ 

Let $S$ be the finite set of primes at which $\pi$ ramifies (including the archimedean place).
For each rational prime $p$ and a choice of an embedding $\iota: E \inj \Qpbar$ there
is a continuous semi-simple representation
\begin{equation}
 \rho_{\pi, p, \iota}: \Gal(\Qbar/\Q) \rightarrow \GSp_4(\Qpbar),\label{galrep1}
\end{equation}
 characterized up to isomorphism by:
 \begin{itemize}
\item  $\rho_{\pi, p, \iota}$ is unramified outside $S \cup \{p\}$;
\item  for each $\ell\not\in S \cup \{p\}$ the characteristic polynomial of $\rho_{\pi, p, \iota}(\Frob_\ell)$  matches with the local $L$-factor of the
local component $\pi_\ell$. More precisely, let  $P_\ell(X) = (1 - \alpha_\ell X)(1 - \beta_\ell X)(1 - \gamma_\ell X)(1 - \delta_\ell X)$ be the Hecke polynomial of $\pi'_\ell$ (defined via its Hecke parameters $\textrm{diag}(\alpha_\ell, \beta_\ell, \gamma_\ell, \delta_\ell) \subset \GSp_4(\C)$). Then  $P_\ell(X) \in E[X]$, and one has
$$ \det(1 - \rho_{\pi, p, \iota}(\Frob_\ell) X) = \iota(P_\ell(X)),$$
where
$$P_\ell(\ell^{-s})^{-1} =: L(\pi_\ell, s - \kappa/2) ,$$
 is the $L$ factor of $\pi_\ell$ (with algebraic normalization involving the shift by $\kappa/2 := (a+b)/2 $). 
\end{itemize}

Choosing an embedding $\iota: E \inj \Qpbar$ determines a prime $\p|p$ in $E$. By enlarging $E$ if necessary, the Galois representations \eqref{galrep1} can be assumed to take values in $E_\p$. We shall thus denote $\rho_{\pi, p, \iota}$ by $\rpp$ henceforth to emphasize the fact that they form a compatible system, as $\p$ varies, in the sense of Serre. 

The compatible system $\{\rho_{\pi, \p}\}_\p$ satisfies a more elaborate local-global compatibility even at the bad primes $S \cup \{p\}$. (This full local-global compatibility plays a crucial role in the present work.) To state this, we first introduce the archimedean $L$-parameter $\phi_{(w; a, b)}$ to which $\pi_\infty$ must belong.
Let $w+1 \equiv a+b \mod 2$. The $L$-parameter $\phi_{(w; a, b)}: W_{\R} = \C^\times \sqcup \C^\times j \rightarrow \GSp_4(\C)$ takes the form $$ z \mapsto |z|^{-w}\cdot \begin{pmatrix} & (z/\overline{z})^{\frac{a+b}{2}} & & & \\ & & (z/\overline{z})^{\frac{a-b}{2}} & & \\
& & &(z/\overline{z})^{-\frac{a-b}{2}} & \\ & & & & (z/\overline{z})^{-\frac{a+b}{2}} \end{pmatrix} $$
on $\C^\times$, and $$ \phi_{(w; a, b)}(j) = \begin{pmatrix} & & & & 1 \\ & & &1 & \\
& &(-1)^{w+1} & & \\ &(-1)^{w+1} & & & \end{pmatrix}.$$

Now, fix isomorphisms $\iota_p: \Qpbar \overset{\sim}{\longrightarrow} \C$ consistent with
the two  embeddings  $E \inj \Qpbar$ and $E \subseteq \C$.

\begin{thm}(Weissauer, Urban, Sorensen, Mok)
Suppose $\pi$ is globally generic and that $\pi  \otimes || \cdot||^{w/2}$ is unitary for some $w \in \Z$. Assume that $\pi_{\infty}$ belongs to the (essentially) discrete series $L$-parameter $\phi_{(w; a, b)}$ with weights $a > b > 0$. Then for any prime $\p$ the
Galois representation $\rpp$ satisfies
 the following properties:

\begin{enumerate}[label = \roman*.)]
\item $\rho_{\pi, \p}$ is unramifed outside $S \cup\{p\},$ where $S$ is the finite set of rational primes at which $\pi$ is ramified together with the archimedean place $\infty$. 
\item for every $\ell \neq p$, full local-global compatibility is satisfied:
\begin{equation} \iota_{p}\, WD (\rppl)^{Fr-ss} = \textrm{rec}_{GT}(\pi_\ell \otimes |c|^{-3/2}), \label{lg2} 
\end{equation} where $\textrm{rec}_{GT}$ is the local Langlands correspondence for $\GSp(4)$ established by Gan-Takeda \cite{GT}.  Here
$c: \GSp_4 \rightarrow \GL_1$ is the similitude character. 

\item $\rpp$ restricted to ${G_{\Q_p}}$ is de Rham with Hodge--Tate weights $\{\delta, \delta + b, \delta+a, \delta+a+b \}$, where $\delta = \frac{1}{2}(w+3 - a-b)$.

\item If $\pi_p$ is unramified, then $\rpp$ is crystalline and we have 
$$\iota_p \det(1 - \varphi X|_{{\bf D}_{cris}(\rpp)})  = P_p(X),$$ where $\varphi$ is the crystalline Frobenius.

\item $\rpp$ is pure of weight $\kappa := w+3$,
i.e., the eigenvalues of $\iota_{p}\,\rpp(\Frob_\ell)$ have absolute value $\ell^{\kappa/2}$ for all $\ell
\not\in S\cup\{p\}$.
\item The similitude character of $\rpp$ is odd: $c\circ \rpp(c_{\infty}) = -1$, for any complex conjugation $c_{\infty} \in G_{\Q}$; furthermore $$ c \circ \rpp = \omega_{\pi^\circ}^{-1}\,\nu_p^{-a-b}.$$ Here $\omega_{\pi^\circ}$ is the finite order character corresponding to the central character of $\pi^\circ$ and $\nu_p$ is the $p$-adic cyclotomic character.
\end{enumerate} 
\label{mainthm} \end{thm}

\begin{proof}
\cite[\S 1]{Sor} and \cite[Theorem 3.1]{Mok}.
\end{proof}

\begin{remark}
It is not necessary to assume that $\pi$ is globally generic in the above theorem. The weaker assumption that $\pi$ belongs to a global generic Arthur parameter will suffice \cite[\S 3.2]{Mok}. Hence it is enough to assume that the local $L$-packet containing $\pi_v$ contains a generic element for \textit{all} places $v$. In \cite{Sor}, global genericity of $\pi$ is used to get a strong lift $\pi \rightsquigarrow \Pi$ from $\GSp_4$ to $\GL_4$ compatible with the local Langlands correspondences for the two groups at all places. In \cite{Mok}, Mok relaxes this assumption on $\pi$ using Arthur's results.
\end{remark}

As remarked earlier, $\rho_{\pi, \p}$ takes values in $E_{\p}$. Let $\O$ be the valuation ring of $E_{\p}$, and $k_{\p}$ its residue field. By choosing a
Galois invariant $\O$-lattice $\Lambda$  in the representation space of $\rpp$ and reducing mod $\p$ we get a representation 
\begin{equation}
\rbpp : G_{\Q, S\cup\{p\}} \rightarrow \GSp_4(k_{\p}) \label{red}.
\end{equation}
 The reduction $\rbpp$ could depend on the choice of the lattice $\Lambda$ but its semi-simplification is independent of it. In particular, such is the case for
$\rbpp$ irreducible.  It is known ( \cite{BLGGT} in the regular weight case and \cite{PSW} more generally) that $\rbpp$ is irreducible for
at least a density one set of primes, assuming that $\pi$ is non-endoscopic and non-CAP.

\section{Deformation theory of Galois representations}

We briefly recall the deformation theory of Galois representations as developed in \cite{Tilouine}.
Consider a residual Galois representation
$$\rhob : \GQSp \to \GSp_4(k);$$
here $k$ is a finite field of characteristic $p$ and $S$ is a finite set of primes.  We assume that $\rhob$ is absolutely
irreducible.  Let $\Ca$ denote the category of inverse limits of artinian local rings with residue field $k$.  For $A \in \Ca$, a {\it lifting}
of $\rhob$ to $A$ is a homomorphism $$\rho : \GQSp \to \GSp_4(A)$$ which yields $\rhob$ on composition with the reduction map $A \to k$.
Two such liftings are considered to be equivalent if one can be conjugated to the other by an element of $\ker(\GSp_4(A) \rightarrow \GSp_4(k))$.  A {\it deformation} of $\rhob$ is then an equivalence class of liftings.  This defines a functor
$$D_{\rhob} : \Ca \to \Sets$$
sending a ring $A \in \Ca$ to the set of deformations of $\rhob$ to $A$.

The functor $D_{\rhob}$ is known to be representable: there is a ring $\Ru \in \Ca$ and a deformation $\rhou \in D_{\rhob}(\Ru)$
identifying $D_{\rhob}$ with the functor $\Hom_{\Ca}(\Ru,-)$: the map
$$\Hom_{\Ca}(\Ru,A) \to D_{\rhob}(A)$$
sending $f : \Ru \to A$ to $f \circ \rhou$ is a bijection for any $A \in \Ca$.

The universal deformation ring $\Ru$ has a non-canonical presentation
$$\Ru \cong W(k)[[T_1,\ldots,T_{d_1}]] / (f_1,\ldots,f_{d_2})$$
where $W(k)$ denotes the Witt vectors of $k$ and
$$d_i := \dim_k H^1(\GQSp, \ad \rhob).$$
Here $\ad \rhob$ denotes the Lie algebra $\gsp_4(k)$ endowed with a $\GQSp$ action via $\rhob$ and conjugation.
Furthermore, from the global Euler characteristic formula  
$$ \frac{|H^1(\GQSp, \ad\rhob )|}{|H^0(\GQSp, \ad\rhob )|\cdot |H^2(\GQSp, \ad\rhob )|} = \frac{|\ad\rhob|}{|(\ad \rhob)^{c = 1}|},$$
a quick calculation (and an application of Schur's lemma) shows that, for $\rhob$ absolutely irreducible
$$d_1 - d_2 = \begin{cases} 1 & \rhob \text{~even} \\
5 & \rhob \text{~odd I} \\
7 & \rhob \text{~odd II}
\end{cases}$$
where $\rhob$ is said to be {\it even} (resp.\ {\it odd I}, resp.\ {\it odd II}) if the image of (any) complex conjugation $c$ is
scalar (resp.\ conjugate to $\diag(1, -1, -1, 1)$, resp.\ conjugate to $\diag(1, -1, 1,-1)$).

We say that the deformation theory of $\rhob$ is {\it unobstructed} if $d_2 = 0$, in which case
$\Ru$ is isomorphic to a power series ring in $d_1$ variables over $W(k)$, with $d_1$ equal to $1,5,7$ as above. 
(We remark that automorphic $\rhob$ such as the ones we are considering fall under the odd II case above by Theorem \ref{mainthm} part vi.)

Using the self-duality of $\ad \rhob$, a  portion of the Poitou--Tate exact sequence yields an exact sequence
\begin{multline*}
0 \to \ker \Bigl[ H^1\bigl (\GQSp,\ad \rhob(1) \bigr) \to \underset{\ell \in S \cup\{p\}}{\bigoplus} H^1\bigl(\Gl, \ad\rhob(1) \bigr) \Bigr]^{\vee} \to \\
H^2(\GQSp,\ad \rhob) \to \ker \Bigl[\underset{\ell \in S \cup \{p\}}{\bigoplus} H^0\bigl(\Gl,\ad \rhob (1) \bigr)^{\vee} \to H^{0}\bigl(\GQSp,\ad \rhob (1) \bigr) \Bigr] \to 0.
\end{multline*}
We define the {\it local obstruction space}
$$\Hl(\GQSp,\ad \rhob) := \ker \Bigl[ \underset{\ell \in S \cup \{p\}}{\bigoplus} H^0\bigl(\Gl,\ad \rhob (1) \bigr)^{\vee} \to H^{0}\bigl(\GQSp,\ad \rhob (1)) \Bigr]$$
and the {\it global obstruction space}
$$\Hg(\GQSp,\ad \rhob) := \ker \Bigl[ H^1\bigl (\GQSp,\ad \rhob(1) \bigr) \to \underset{\ell \in S \cup \{p\}}{\bigoplus} H^1\bigl(\Gl, \ad\rhob(1) \bigr) \Bigr]^{\vee}.$$
Note that these definitions make sense without any irreducibility assumption on $\rhob$.

The global obstruction space is (dual to) a Selmer group for $\ad \rhob(1)$ with strict local conditions at all ramified primes (including the coefficient primes), and 
thus is controlled by any result controlling such Selmer groups \cite{GeT}. Our focus in this paper is on the local obstruction space
$\Hl(\GQSp, \ad \rhob)$.

\section{Unobstructedness for Modular Forms}
We quickly review the proof of the analogue of our theorem for an automorphic representation $\pi$ of $\GL(2)$ corresponding to a classical modular form $f$, as in \cite{Weston}.  As with $\GSp(4)$, the local portion of the proof quickly reduces to showing that for any fixed prime $\ell$, the twisted adjoint invariants $H^0(G_\ell,\ad \rbpp (1))$ vanish for all but finitely many coefficient primes $\p$; and that $H^0(G_p,\ad \rbpp(1))$ vanishes for all but finitely many $\p$.  (Here $p$ is the residue characteristic of $\p$.)  The latter case is dealt with using any description of $\rbpp|_{G_p}$ coming from an integral version of $p$-adic Hodge theory.  For the former case, since we are varying the coefficient prime, the first essential input is a strong local--global compatibility result which allows one to describe $\rbppl$ in terms of the type of $\pi_\ell$.  The principal series and supercuspidal cases can then be proved using purely local methods.  By contrast, we know of no purely local approach in the Steinberg case.  Instead, we are forced to rely on global level lowering results to get a precise enough description of $\rbppl$.

The proof for $\GSp(4)$ follows essentially the same pattern. The key assumption is that $\pi$ is everywhere locally generic, which is needed for the Steinberg case. 

We remark that to prove the vanishing of $H^2_{\mathrm {glo}}$, by \cite[Proposition 2.2]{Weston} it suffices to show the finiteness of the Selmer group $H^1_f (\GQSp, \ad\rpp (1))$ and the vanishing of the Selmer group $H^1_f(\GQSp, \ad \rpp \otimes_{\mathcal O} E_\p/\mathcal O)$.  Such results can in some cases perhaps be deduced from results in Genestier-Tilouine \cite{GeT}.

\section{Non-supercuspidal case}

\subsection{Statements}

  Our goal in this section is to prove the following result.

\begin{prop} \label{prop:nsc}
Let $\pi$ be as in the introduction.  Let $\ell$ be a prime for which the local component $\pi_\ell$ is not supercuspidal.  Then
$$H^{0}\bigl(\Gl,\arb(1) \bigr) = 0$$
for all but finitely many primes $\p$ of $\Kp$.
\label{nss}
\end{prop}

We follow the classification
of \cite[Section 2.4]{RobertsSchmidt}, which gives explicit descriptions of the Weil--Deligne representations corresponding to
non-supercuspidal representations of $\GSp_4(\Ql)$.  These in turn allow us to compute $\rppl$ for all $\p \nmid \ell$.  However, even
though $\rbppl$ is a well-defined object when the global residual representation $\rbpp$ is irreducible, passing directly from
$\rppl$ to $\rbppl$ is more ambiguous.  Instead, all that is initially clear is that the semisimplification $\rbppls$ is the reduction of $\rppls$. 
In order to move beyond this, we invoke the following global result on level lowering.

\begin{thm} \label{thm:ll}
Let $\pi$ be as above.  Fix a prime $\p$ of residue characteristic $p \neq \ell$ such that the residual representation $\rbpp$ is irreducible.  Assume further that $p > a+b+3$
and that $\pi$ is unramified at $p$.  Then there exists a globally generic
automorphic representation $\pi'$ of $\GSp(4)$ (essentially discrete series with the same weights as $\pi$)
 and a prime ideal $\p'$ of residue characteristic $p$ of the field of coefficients $K'$ of $\pi'$ such that
\begin{enumerate}
\item The paramodular level of $\pi'$ is equal to the Artin conductor of $\rbpp$;
\item The residual representations $\rbpp$ and $\rhob_{\pi',\p'}$ are isomorphic over $\Fpbar$.
\end{enumerate}
\end{thm}
\begin{proof}
This is essentially \cite[Thm 7.6.6]{GG} on the existence of optimal lifts.  Note that the ordinary hypothesis is unnecessary under the assumption that
$p > a+b+3$, so that $\rbpp$ is in the Fontaine--Lafaille range. The needed assumptions on the image of $\rbpp$ can be relaxed to adequacy of $\rbpp(G_{\Q(\zeta_p)})$, which follows from \cite[Proposition 5.3.2]{BLGGT}.
(Note that the proof does guarantee that the resulting lifts are unramified outside the
specified set, although this was left out of the statement of the theorem.)
\end{proof}

\subsection{Reduction}

Let $V$ be a four-dimensional representation of $\Gl$ over $\Fpbar$.
We say that $V$ is {\it generic} if there exists  a Jordan--Holder series
$$0 =V_0 \subsetneq V_1 \subsetneq \cdots \subsetneq V_d \subsetneq V$$
with the property that for any $i$ such that there exists $j$ with
\begin{equation} \label{eq:nt}
V_i/V_{i-1} \cong V_j/V_{j-1}(1),
\end{equation}
 then in fact 
$$V_i/V_{i-1} \cong V_{i+1}/V_{i}(1)$$
and $V_{i+1}/V_{i-1}$ represents a non-trivial element in
$$\Ext^{1}_{\Fpbar[\Il]} \left( V_{i+1}/V_i, V_{i}/V_{i-1} \right).$$

We note that this definition may not be optimal for groups other than $\GSp(4)$, but it will suffice for our purposes.  It is likely
that the notion we give here is closely related to that discussed in \cite[Definition 1.1.2]{Allen}, but we have not pursued the connection here.

\begin{prop} \label{prop:red}
Let $\pi$ be as above and let $\ell$ be a prime such that $\pi_\ell$ is not supercuspidal.  Then $\rbppl$ is generic for all but finitely many $\p$.
\end{prop}

One could (but probably shouldn't) paraphrase the above result as saying that globally generic automorphic
represntations give rise to generically residually generic Galois representations.

\begin{proof}
Unfortunately, we do not presently have a conceptual proof of this proposition.  Instead we are forced to resort to a case study using \cite{RobertsSchmidt}.  The result is
vacuous for Groups I, VII, VIII and X as (\ref{eq:nt}) never occurs.  Of the remaining cases, Groups II, IV, V and VI are quite similar.
We will thus content ourselves with a discussion of Groups III, IV, IX and XI.  Since we are always content to eliminate finitely many $\p$, we can assume
we are in a situation  where Theorem~\ref{thm:ll} applies.

In Group IV, we have
$$\rbppls \cong \sp\cycp^{3/2} \oplus \sp\cycp^{1/2} \oplus \sp\cycp^{-1/2} \oplus \sp\cycp^{-3/2}$$
for some character $\sigma : \Ql^{\times} \to \C^{\times}$.
We thus must show that each of the succesive extensions above is intertially non-trivial for all but finitely many $\p$.  
Since the property of residual genericity is invariant under twisting by a character,
choosing a Dirichlet character $\tilde{\sigma}$ extending $\sigma$ and replacing $\pi$ by $\tilde{\sigma}^{-1} \otimes \pi$
we may assume that $\sigma$ is trivial.
In this case we have that the exponent of $\ell$ in the Artin conductor of $\rbppl$ is precisely the number of these extensions which
are inertially non-trivial.  In particular, if it is not $\ell^3$, then by Theorem~\ref{thm:ll} there exists a congruence mod $\p$
between $\pi$ and some other automorphic representation of smaller level.  Since there are only finitely many essentially discrete series automorphic
representations of smaller level and each can only be congruent to $\pi$ modulo finitely many $\p$, it follows that
the Artin conductor of $\rbppl$ is $\ell^3$ for all but finitely many $\p$, as desired.

For Group III, as above we may twist away the character $\sigma$ to assume that
$$\rbppls \cong \cp\cycp^{1/2} \oplus \cp\cycp^{-1/2} \oplus \cycp^{1/2} \oplus \cycp^{-1/2}$$
for some character $\chi : \Ql^{\times} \to \C^{\times}$.  However, in this case the only way for
$\rbppl$ to have symplectic image is if the two subextensions above are either simultaneously split or non-split.
Using this one argues as in the Group IV case with little additional change.

For Group IX, we have that
$$\rbppls \cong \xi_\p \cycp^{1/2} \det \mu_\p \mu_\p' \oplus \cycp^{-1/2} \mu_\p$$
for an irreducible two-dimensional $\mu : \Gl \to\GL_2(\Fpbar)$ and a non-trivial quadratic character $\xi_p$ such that
$\xi_p \mu_p = \mu_p$.  However, $\rppl$ is non-semisimple with Artin conductor $\ell^{2N+1}$, with $\ell^N$ the
Artin conductor of $\mu$.  In particular, when $\rbppl$ has Artin conductor $\ell^{2N}$, it follows from level lowering
as above that there exists an automorphic representaion of smaller level congruent to $\pi$ modulo $\p$.  As before, this
can occur only for finitely many $\p$.

In Group XI we may again twist away the character $\sigma$ in order to assume that
$$\rbppls \cong \cycp^{1/2} \oplus \mu \oplus \cycp^{-1/2}$$
with $\mu : \Gl \to \GL_2(\Fpbar)$ the reduction of an irreducible representation.  Eliminating finitely many $\p$, we may assume
that $\mu$ is itself irreducible.  It follows that the Artin conductor of $\rbppls$ is equal to the Artin conductor of $\mu$, plus one
if the extension of the characters is non-trivial.  From here the argument proceeds as before.
\end{proof}

\begin{remark}
The finitely many $\p$'s that needs to be discarded in the above proposition include congruence primes of $\pi$, and primes $\p$ for which the (paramodular) Hecke eignevalues of a ramified component $\pi_v$ of $\pi$ satisfy certain congruences mod $\p$. For $\pi$ which has Iwahori level, this exceptional set of primes will be made precise in \cite{Mullath}.
\end{remark}

\subsection{Steinberg-type representations}

It will be convenient to introduce notation for a class of representations of $\Gl$ which occur below.  Fix a
Frobenius element $F \in \Gl$.  For a homomorphism $t : I_\ell \to \kp$, 
we extend $t$ to a function $t : \Gl \to \kp$ by setting $t(F^n i)=t(i)$ for $i \in I_\ell$.  Given a semisimple representation
$\sigma : \Gl \to \GL_n(\kp)$ and a nilpotent matrix $N$ such that $\sigma N \sigma^{-1} = \cycp N$, define a
representation
$$\rho_t(\sigma,N) : \Gl \to \GL_n(\kp)$$
by
$$\rho_t(g) = \sigma(g) \exp(Nt(g)).$$
Note that
$$H^0\bigl(\Gl,\rho_t(\sigma,N) \bigr) = \{ v \in k^n \mid \sigma(F)v=v \text{~and~} \sigma(i)\exp(Nt(i))v =v \text{~for all } i \in I_\ell \}.$$

We will specifically make use of the generalized Steinberg representations
$$\St_2(\tp) = \rho_{\tp} \bigl( \cycp^{1/2} \oplus \cycp^{-1/2}, N_1 \bigr)$$
$$\St_3(\tp) = \rho_{\tp} \bigl(\cycp\oplus \trr \oplus \cycp^{-1},N_2 \bigl)$$
$$\St_{2,2}(\tp) = \rho_{\tp} \bigl(\cycp\oplus \trr \oplus \trr \oplus \cycp^{-1},N_3 \bigl)$$
$$\St_{4}(\tp) = \rho_{\tp} \bigl(\cycp^{3/2}\oplus \cycp^{1/2} \oplus \cycp^{-1/2} \oplus \cycp^{-3/2},N_4 \bigl)$$
with
$$N_1 = \left( \begin{array}{cc} 0 & 1 \\  0 & 0 \end{array} \right); \qquad
N_2 = \left( \begin{array}{ccc} 0 & -2 & 0 \\  0 & 0 & 1 \\ 0 & 0 & 0 \end{array} \right);$$
$$N_3 = \left( \begin{array}{cccc} 0 & 1 &  0 & 0 \\  0 & 0 & 0 & 0 \\ 0 & 0 & 0 & -1 \\  0 & 0 & 0 & 0 \end{array} \right)
N_4 = \left( \begin{array}{cccc} 0 & 1 &  0 & 0 \\  0 & 0 & 1 & 0 \\ 0 & 0 & 0 & -1 \\  0 & 0 & 0 & 0 \end{array} \right).$$
One checks easily that all of these parameters have image in $\GSp_4(\kp)$.
Using the description of $H^0$ given above, one immediately checks that
for  $\chi : \Gl \to \kp^{\times}$ any character, we have
$$H^{0}\bigl(\Gl,\St_2(\tp) \otimes \chi \bigr) = H^0\bigl(\Gl,\kp(\chi\cycp^{1/2}) \bigr);$$
$$H^{0}\bigl(\Gl,\St_3(\tp) \otimes \chi \bigr) = H^0\bigl(\Gl,\kp(\chi\cycp) \bigr);$$
$$H^{0}\bigl(\Gl,\St_{2,2}(\tp) \otimes \chi\bigr) = H^0\bigl(\Gl,\kp(\chi\cycp) \bigr).$$
$$H^{0}\bigl(\Gl,\St_{4}(\tp) \otimes \chi\bigr) = H^0\bigl(\Gl,\kp(\chi\cycp^{3/2}) \bigr).$$

\subsection{Proof of the proposition \ref{nss}}

If $\rho : \Gl \to \GL_2(\Kpp)$ is a two-dimensional representation, we write $\ad \rho$ (resp.\ $\ad^0 \rho$) for its four-dimensional
adjoint representation (resp.\ three-dimensional trace zero adjoint representation).

We will
repeatedly use the fact that any reducible two-dimensional representation $\Gl \to \GL_2(\kp)$ with semisimplification $\cycp^{1/2} \oplus \cycp^{-1/2}$ and of
Artin conductor $\ell$ has the form $\St_2(\tp)$ for some
$\tp : I_\ell \to \kp$; this is a straightforward and familiar Galois cohomology calculation.

\subsubsection{Group I}

This is the case where $\rppl$ is diagonal.  One immediately computes that
$$\arl \cong \trr^2 \oplus \cop \oplus \cop^{-1} \oplus \ctp \oplus \ctp^{-1} \oplus \cop\ctp \oplus \cop^{-1}\ctp^{-1} \oplus
\frac{\cop}{\ctp} \oplus \frac{\ctp}{\cop}$$
for characters $\chi_1,\chi_2 : \Ql^{\times} \to \C^{\times}$.  The representation $\artsl$
has non-trivial $\Gl$-invariants if and only if one of the characters above equals $\cycp^{-1}$.  This in turn can only occur for infinitely many $p$ if one of
the characters
$$\trr, \chi_1, \chi_1^{-1}, \chi_2, \chi_2^{-1}, \chi_1\chi_2, \chi_1^{-1}\chi_2^{-1}, \frac{\chi_1}{\chi_2}, \frac{\chi_2}{\chi_1}$$
equals the absolute value character $\nu$.  However, each of these cases corresponds to one of the groups II--VI below and thus does not occur in group I.

\subsubsection{Group II}

In this case we have
$$\rppls \cong \cp^2 \oplus \cp\cycp^{1/2} \oplus \cp\cycp^{-1/2} \oplus 1$$
for a character $\chi : \Ql^{\times} \to \C^{\times}$ with $\chi^2 \neq \nu^{\pm 1}$ and $\chi \neq \nu^{\pm 3/2}$.
By Proposition~\ref{prop:red}, after excluding finitely many $\p$ we may assume that
$$\rbppl \cong \cbp^2 \oplus \cbp \St_2(\tp) \oplus 1$$
for some non-trivial homomorphism $\tp : \Il \to \kp$.  One then computes that
$$\arbl \cong \trr \oplus \cbp^2 \oplus \cbp^{-2} \oplus \cbp\St_2(\tp) \oplus \cbp^{-1}\St_2(\tp) \oplus \St_3(\tp).$$
Thus
$$H^{0}\bigl(\Gl,\arbl(1)\bigr) \cong H^{0}\bigl(\Gl,\cycbp \oplus \cbp^2\cycbp \oplus \cbp^{-2}\cycbp \oplus \cbp\cycbp^{3/2} \oplus
\cbp^{-1}\cycbp^{3/2} \oplus \cycbp^{2} \bigr).$$
By the assumptions on $\chi$, none of these characters has invariants for infinitely many $\p$, as desired.

\subsubsection{Group III}

In this case we have
$$\rppls \cong \cp\cycp^{1/2} \oplus \cp\cycp^{-1/2} \oplus \cycp^{1/2} \oplus \cycp^{-1/2}$$
for a character $\chi : \Ql^{\times} \to \C^{\times}$ with $\chi \neq 1$ and $\chi \neq \nu^{\pm 2}$.
By Proposition~\ref{prop:red} we thus have that for all but finitely $\p$
$$\rbppl \cong \cbp\St_2(\tp) \oplus \St_2(-\tp)$$
for some non-trivial homomorphism $\tp : \Il \to \kp$.  (The opposite signs in the Steinberg representations are required to make the image symplectic.)
One computes that
$$\arbl\cong \cbp \St_3(\tp) \oplus \cbp^{-1} \St_3(-\tp) \oplus \St_{2,2}(\tp).$$
Therefore
$$H^{0}\bigl(\Gl,\arbl(1) \bigr) \cong H^{0}\bigl(\Gl,\cbp\cycbp^2 \oplus \cbp^{-1}\cycbp^2 \oplus \cycbp^2 \bigr).$$
As before, this vanishes for all but finitely many $\p$ by the assumptions on $\chi$.

\subsubsection{Group IV}

In this case we have
$$\rppls \cong \cycp^{3/2} \oplus \cycp^{1/2} \oplus \cycp^{-1/2} \oplus \cycp^{-3/2}$$
Thus for all but finitely many $\p$ we have
$$\rbppl \cong \St_4(\tp)$$
for some non-trivial homomorphism $\tp : \Il \to \kp$. 
A straightforward computation then shows that
$$H^0\bigl(\Gl,\arb(1) \bigr) = H^0\bigl(\Gl,\cycbp^4\bigr)$$
which vanishes for all but finitely many $\p$.

\subsubsection{Group V}

In this case we have
$$\rppls \cong \cycp^{1/2} \oplus \cycp^{1/2}\xi_\p \oplus \cycp^{-1/2}\xi_\p \oplus \cycp^{-1/2}$$
for a non-trivial quadratic character $\xi : \Gl \to \C^\times$.  For all but finitely many $\p$ we thus have
$$\rbppl \cong \St_2(\tp) \oplus \xi\St_2(\tp)$$
for some non-trivial $\tp : \Il \to \kp$.  It follows that
$$\arb \cong \St_3(\tp)^2 \oplus \xi_\p\St_{2,2}(\tp),$$
so that
$$H^{0}\bigl(\Gl,\arb(1)\bigr) \cong H^{0}\bigl(\Gl,\cycbp^2 \oplus  \cycbp^2 \oplus \cycbp^2\xi_\p \bigr)$$
which vanishes for all but finitely many $\p$.

\subsubsection{Group VI}

This case is identical to Group V except without the quadratic character $\xi$.  The argument proceeds as before.

\subsubsection{Group  VII}

In this case we have
$$\rppls \cong \chi_\p \wpp \rss' \oplus \rss$$
for a supercuspidal representation $\psi$ of $\GL_2(\Ql)$ with associated Galois representation
$\rss : \Gl \to \GL_2(\Kp)$ and central character $\wp$ together with
 a character $\chi : \Ql^{\times} \to \C^{\times}$.
One computes that
$$\ar \cong \ad \rho_{\psi,\p} \oplus  \chi_\p\ad^0 \rho_{\psi,\p} \oplus \chi_\p^{-1}\ad^0 \rbss.$$
Each of these subrepresentations has generically trivial invariants after twisting by any character by \cite{Weston}.

\subsubsection{Group VIII}

This is identical to the preceeding case except without the character $\chi$.  The argument proceeds as before.

\subsubsection{Group IX}
Fortunately, for Group IX it suffices to consider the semi-simplification, which is identical to Group VII with $\chi$ equal to the product of $\nu$ and a quadratic character.

\subsubsection{Group X}

In this case we have
$$\rppls \cong \wpp \oplus \rss \oplus \trr$$
for a supercuspidal representation $\psi$ of $\GL(2)$ with $ \wp \neq \nu^{\pm 1}$.
It follows that
$$\arb \cong \trr \oplus \wpp \oplus \wpp^{-1} \oplus \rbss^{\tr} \oplus \wpp^{-1} \rbss^{\tr} \oplus \ad^0 \rbss$$
Since $\rbss$ is generically irreducible, the generic vanishing follows again from \cite{Weston} and the fact that $\wpp \neq \nu^{\pm 1}$.

\subsubsection{Group XI}

We now have
$$\rppls \cong \cycp^{1/2} \oplus \cycp^{-1/2} \oplus \rbss$$
for a supercuspidal representation $\psi$ of $\GL(2)$ with $\wp = 1$.  By Proposition~\ref{prop:red} we thus have
$$\rbppls \cong \St_2(\tp) \oplus \rbss$$
for a non-trivial homomorphism $\tp : \Il \to \Kp$ for all but finitely many $\p$.
In this case we have
$$\arb \cong \St_3(\tp) \oplus \ad^0 \rbss \oplus \cycbp^{-1/2} \left[ \begin{array}{cc} \cycbp \rbss^{\tr} & \tp \rbss^{\tr} \\
0 & \rbss^{\tr} \end{array} \right].$$
Even the semi-simplifcation of this has generically trivial Galois invariants after a cyclotomic twist, completing the proof.

\section{Supercuspidal case}

 Recall that the Galois group $\Gl$
has a filtration
$$P_\ell \triangleleft I_\ell \triangleleft \Gl$$
(here $P_\ell$ denotes the wild inertia group) such that $P_\ell$ is pro-$\ell$,
$$I_{\ell}/P_\ell \cong \Zhat / \Z_\ell \qquad \text{and} \qquad \Gl/I_{\ell} \cong \Zhat.$$
It follows that for any finite Galois extension $K/\Ql$, there is a corresponding filtration
$$P(K/\Ql) \triangleleft I(K/\Ql) \triangleleft \Gal(K/\Ql)$$
with $P(K/\Ql)$ an $\ell$-group, $I(K/\Ql)/P(K/\Ql)$ cyclic of order prime to $\ell$ and $\Gal(K/\Ql)/I(K/\Ql)$ cyclic.
Note in particular that none of the groups $A_n$ or $S_n$ for $n \geq 4$ are quotients of $\Gl$ for $\ell \neq 2$.
(This is clear for $n \geq 5$ since then $A_n$ is not even solvable.  For $A_4$, it results from observing that $A_4$ has
no such filtration.)

\begin{prop} \label{prop:subgp}
Assume that $\ell \neq 2$.
Let $\rho : \Gl \to \GSp_4(\C)$ be an irreducible representation with finite image.  Then there is a quartic Galois extension $M/\Ql$ with
Galois group $\Z/2 \times \Z/2$  and a character $\chi : G_M \to \C^{\times}$ such that
$\rho \cong \Ind^{\Gl}_{G_M} \chi$.
\label{sc}
\end{prop}
\begin{proof}
Let $V = \C^4$ endowed with a $\Gl$-action via $\rho$.  Let $G$ denote the image of $\rho$.  Let $d$ denote the largest divisor of $4$ with the
following property: there exist a direct sum decomposition $V = W_1 \oplus \cdots \oplus W_d$ such that $G$ acts on the set
$\{ W_1, \ldots, W_d \}$. 
Note that each $W_i$ is then necessarily of dimension $\frac{4}{d}$.  Further, by irreducibility
$G$ acts transitively on $\{W_1, \ldots, W_d \}$.
We consider the three possible values of $d$ separetely.

If $d=1$, then $\rho$ is said to be a {\it primitive} representation.  The primitive represenations of $\GSp_4$ have been classified; see for example
\cite{Martin}.  However, each group which occurs as the image of a primitive representation of $\GSp_4$ has $(\Z/2)^4$ as a quotient.
By the structure given for quotients of $\Gl$ preceeding the proposition, this is impossible for $\ell \neq 2$.

If $d=2$, the action of $G$ on $\{ W_1, W_2 \}$ yields an exact sequence
$$1 \to H \to G \to S_2 \to 1$$
with $H$ the common stabilizer of $W_1$ and $W_2$.  The representation $\rho : H \to \GL(W_1)$ is itself primitive (since otherwise we would
have $d=4$).  The primitive subgroups of $\GL_2(\C)$ are beyond well known: they have projective image $A_4$, $S_4$ or $A_5$.
As observed above, none of these groups is a quotient of $\Gl$, so this case can not occur.

If $d=4$, the action of $G$ on $\{ W_1,W_2,W_3,W_4\}$ yields an exact sequence
$$1 \to H \to G \to S \to 1$$
with $S$ a transitive subgroup of $S_4$.  As observed above, $S$ can not be $A_4$ or $S_4$.  It is therefore either cyclic of order $4$,
the subgroup $V \cong \Z/2 \times \Z/2$, or a $2$-Sylow subgroup of $S_4$ (which is isomorphic to the dihedral group $D_4$).
In each case let $H'$ denote the stabilizer of $W_1$ (this is $H$ in the first two cases and a subgroup of $G$ containing $H$ with index $2$ in
the last case) and let $\chi : H' \to \C^{\times}$ denote the character giving the action of $H'$ on $W_1$.  Then one checks easily that
$\rho \cong \Ind_{H'}^{G} \chi$.

However, such a representation a priori take values in $\GL_4(\C)$, not $\GSp_4(\C)$.  A simple explicit computation with the Weyl group of
$\GSp_4$ allows one to determine that only in the case $S=V$ can the representation be conjugated to land in $\GSp_4(\C)$.
\end{proof}

\begin{prop} \label{prop:sc}
Assume that $\ell \neq 2$.  Let $\pi_\ell$ be a supercuspidal representation of $\GSp_4(\Ql)$.  Then
$$H^0(\Gl,\arb (1) ) = 0$$
for all but finitely many primes $\p$.
\end{prop}
\begin{proof}
Let $\mu : W_{\Ql} \to \GSp_4(\C)$ be the Weil-Deligne representation associated to $\pi_\ell$ by \cite{GT}; it is either irreducible (as a $\GL_4(\C)$ representation) or
of the form $\mu_1 \oplus \mu_2$ with $\mu_i$ irreducible two-dimensional representations of $\Gl$
with $\det \mu_1 = \det \mu_2$. (In both cases the monodromy $N = 0$.) After twisting $\mu$ by a power of the norm character, we may assume that $\mu$ has finite image. Hence we may regard $\mu$ as a representation $\Gl \to \GSp_4(\C)$ as in proposition \ref{sc}.

Let $M$ denote the biquadratic extension of $\Ql$ from proposition \ref{sc}.  For an element $\sigma \in \Gal(M/\Ql)$ and a character
$\chi : G_M \to \C^{\times}$, we write $\chi^{\sigma}$ for the character given by
$$\chi^{\sigma}(g) = \chi(\tilde{\sigma} g \tilde{\sigma}^{-1})$$
for some choice of $\tilde{\sigma} \in \Gl$ lifting $\sigma$; the definition of $\chi^{\sigma}$ is independent of this choice.
We will in fact show that $H^0(G_M,\arb(1))$ vanishes for all but finitely
many $\p$.

In the irreducible case, Proposition~\ref{prop:subgp} shows that $\mu$ is isomorphic to $\Ind_{M}^{\Ql} \chi$ for
a character $\chi : G_M \to \C^{\times}$.
It follows that for every prime $\p$ we have
$$\rbppls \cong \Ind_{M}^{\Ql} \bar{\chi}_\p.$$
In particular, $$\rbpp^{\text{ss}}|_{G_M} \cong \underset{\sigma \in \Gal(M/\Ql)}{\oplus} \chi^{\sigma}_\p.$$
The representation $\arb|_{G_M}$ is thus simply a sum of characters of the form $\chi^{\sigma}_\p/\chi^{\tau}_\p$ for
$\sigma,\tau \in \Gal(M/\Ql)$.  It thus suffices to show that $\chi^{\sigma}/\chi^\tau \neq \nu$ for any
$\sigma,\tau \in \Gal(M/\Ql)$.   However, this is easily seen since the action of $\Gal(M/\Ql)$ on $G_M^{\ab}$ is trivial modulo
inertia.

The argument in the other case is similar: we now have
$$\rbppls \cong \Ind_{M_1}^{\Ql} \bar{\chi}_{1,\p} \oplus \Ind_{M_2}^{\Ql} \bar{\chi}_{2,\p}$$
for characters $\chi_1 : G_{M_1} \to \C^{\times}$ and $\chi_2 : G_{M_2} \to \C^{\times}$ with $M_1$ and $M_2$ quadratic and
$\chi_1 \cdot \chi_1^{\sigma_1} = \chi_2 \cdot \chi_2^{\sigma_2}$ where $\sigma_i \in \Gal(M/\Ql)$ generates $\Gal(M_i/\Ql)$.
Thus $\arb|_{G_M}$ is simply a sum of ratios of $\Gal(M/\Ql)$-conjugates of $\chi_{1,\p}$ and $\chi_{2,\p}$.  As before, there is certainly no
way such a ratio with the same character can equal $\nu$, so we are left to consider a ratio
$\chi^{\sigma}_{1}/\chi^{\tau}_{2}$.  Acting by $\sigma^{-1}$, we may assume that $\sigma = 1$.

Suppose therefore that $\chi_1/\chi^{\tau}_2 = \nu$.  Acting by $\sigma_1$ and multiplying, we obtain
$$\nu^2 = \frac{\chi_1\chi_1^{\sigma_1}}{\chi_2^{\tau}\chi_2^{\tau\sigma_1}} = \frac{\chi_2\chi_2^{\sigma_2}}{\chi_2^{\tau}\chi_2^{\tau\sigma_1}}.$$
Since at least one factor in the numerator cancels at least one factor in the denominator, we are left in the situation of ratios of conjugates of a single
character, as before.

\end{proof}

\section{The case $\ell = p$}

\subsection{Fontaine-Laffaille theory}
We recall the theory of Fontaine-Laffaille adapted to our situation, following the summary in \cite[\S 1.4]{BLGGT}.
We have at hand a representation $\rho: \Gp \rightarrow \GSp_4(\O)$, where $\O$ is the valuation ring of a finite extension of $\Q_p$ with residue field $k$. In this section, we shall mostly ignore the symplectic structure preserved by $\rho$ and view $\rho$ as a representation valued in $\GL_4(\O)$. This suffices for our purposes.

Let $\MF_\O$ be the category of filtered Dieudonne $\O$-modules. Objects in $\MF_\O$ consist of triples $(D, (D^i)_{i \in \Z}, (f_i)_{i \in \Z})$ where $D$ is a finitely generated $\O$-module, $(D^i)_{i \in \Z}$ is a decreasing filtration of $D$ by $\O$-module direct summands, and $(f_i: D^i \rightarrow D)_{i \in \Z}$ are $\O$-linear maps satisfying:
\begin{itemize}
\item $D^i = D$ for $i  \ll 0$; $D^i = (0)$ for $i \gg 0$;
\item $f_i|_{D^{i+1}} = p\cdot f_{i+1}$ for all $i$;
\item $D = \sum_{i \in \Z} f_i(D^i)$.
\end{itemize}

A morphism $(D, (D^i), (f_i)) \rightarrow (E, (E^i), (g_i))$ in $\MF_\O$ consists of an $\O$-linear map $t: D \rightarrow E$ satisfying $t(D^i) \subset E^i$ and $t \circ f_i =g_i \circ t$ for all $i$. 

We denote by $\MF_k$ the full subcategory of $\MF_\O$ consisting of objects annihilated by $\p$ (the maximal ideal of $\O$).

Let $\Rep_{\O}(\Gp)$ be the category of finitely generated $\O$-modules with a continuous $\Gp$-action. There exists an exact fully faithful (covariant) functor of $\O$-linear categories $${\bf G}: \MF_\O \longrightarrow \Rep_{\O}(\Gp)$$ which commutes with Tate twists and preserves the length (and rank) of the underlying $\O$-modules. The essential image of $\bf G$ is closed under taking sub-objects and quotients. See \cite[\S 2.4.1]{CHT} for a definition of the functor $\bf G$. This functor $\bf G$ restricts to a functor $${\bf G}: \MF_k \longrightarrow \Rep_k(\Gp).$$ 
 
 \begin{definition} For an object $D \in  \MF_k$ let $\WFL(D)$ be the multiset of integers $i$ such that $\textrm{gr}^i (D) := D^{i}/D^{i+1} \neq (0)$, with $i$ being counted with multiplicity equal to the $k$-dimension of this gradation.
 \end{definition}

We list some key properties of the functor $\bf G$ and the category $\MF_k$ that shall be useful for us:
\begin{enumerate}[label = (\Alph*)]
\item (\cite[\S 1.4]{BLGGT}) If $M$ is a $p$-torsion free object of $\MF_\O$, then ${\bf G}(M)\otimes_{\Z_p}\Q_p$ (= ${\bf G}(M)[1/p]$) is a crystalline representation of $\Gp$ with Hodge--Tate weights $$\mathrm{HT}(M) = \WFL(M \otimes_{\O} k).$$
\item (\cite[\S 1.4]{BLGGT}) If $L$ is a $\Gp$-invariant lattice in a crystalline representation $V \in \Rep_{\O}(\Gp)$ with all its Hodge--Tate weights
 in the range $[0, p-2]$, then $L$ (and hence $L/\p L$) is in the image of $\bf G$.  
\item (\cite[\S 1.10 (b)]{FL}) Morphisms in $\MF_k$ are strict with filtrations. Namely, if $t: M \rightarrow N$ is such a morphism, then $$t(M^i) = t(M)\cap N^i.$$
for any $i$. (A proof of this particular version of this statement can be found in \cite[proposition 2.5]{Hattori}). Consequently, if two objects $M, N \in \MF_k$ satisfy $$\WFL(M) \cap \WFL(N)  = \emptyset,$$ then \begin{equation}\Hom_{\MF_k} (M, N) = 0. \label{vanish}\end{equation}
\end{enumerate}

\subsection{Local deformations at $l = p$}

We assume that $p$ is a good prime for $\pi$, i.e., $\pi_p$ is unramified. Let $V_{\pi}$ denote a free rank $4$ $\O$-module equipped
with a $\Gp$-action via $\rpp|_{\Gp}$.  Let $w \in \Z$ be such that the central character of $\pi_{\infty}$ is of the form $c \to c^{-w}$ and let $\delta = \frac{1}{2}(w+3-a-b)$ (where $a \geq  b \geq 0$ are determined by $\pi_{\infty}$)  . It is known (\cite{U}, \cite[Theorem 3.1]{Mok}) that $V_{\pi}$ is crystalline with Hodge-Tate weights

\begin{equation}
\delta + \{0, b, a, a+b \}.\label{wts}
\end{equation}

\begin{prop}
Suppose $0< a+b < p-2$. If $b, a-b \neq 1$, then $$H^0(\Gp, \ad \rbpp(1)) = 0.$$
\end{prop}

\begin{proof}
As $\ad \rbpp(1)$ is a submodule of $\End(\rbpp)(1)$, it suffices to prove that $$\Hom_k(\rbpp, \rbpp(1))$$ has no $\Gp$ invariants. Equivalently we
may show that $\Hom_k(\rbpp(-\delta), \rbpp(1-\delta))$ has no $\Gp$ invariants. So we might as well assume $\delta =0$. After this, we are in the Fontaine-Laffaille range (because $a+b < p-2$). Thus, by the fullness of $\bf G$, it  suffices to prove that 

\begin{equation}\Hom_{\MF_k} ({\bf G}^{-1} (\rbpp |_{\Gp}), {\bf G}^{-1} (\rbpp |_{\Gp})(1)) = 0. \label{vanish2}\end{equation} 

Let $T_{\pi} \subset V_{\pi}$ be a $\Gp$ stable $\O$-lattice such that ${\bf G}(D_{\pi}) = T_{\pi}$. (Such a lattice exists by property (B) of $\bf G$.) In
particular, $${\bf G}(D_{\pi}/\p D_{\pi}) = T_{\pi}/\p T_{\pi} \cong \rbpp|_{\Gp}.$$  We know that the Hodge-Tate numbers of $V_{\pi}$ are given by 
\eqref{wts}. Our assumption on $a$ and $b$ implies that no two Hodge-Tate numbers in \eqref{wts} are consecutive. From this, by property (A), we see that $$\WFL(D_{\pi}/\p D_{\pi}) \cap \WFL(D_{\pi}/\p  D_{\pi}(1)) = \emptyset.$$
The vanishing in \eqref{vanish2} now follows from the vanishing in \eqref{vanish} of property (C).
\end{proof}

\subsection{The Ordinary Case}

We include the following result for the sake of completeness.

\begin{prop}
If $\pi_p$ is unramified and ordinary, then if $H^0(\Gp, \ad \rbpp(1)) \neq 0$ one of the following congruences must hold: 

\begin{align*}
a & \equiv \pm 1 \; \pmod{p-1}; \\
b & \equiv \pm 1\;  \pmod{p-1}; \\
a-b & \equiv \pm 1 \; \pmod{p-1}; \\
a+b & \equiv \pm 1 \; \pmod{p-1}.
\end{align*}
\end{prop}

\begin{proof}
Since $$\dim_k\,H^0(G_{\Q_p}, ad\, \rbpp(1))  \leq \dim_k\,H^0(I_p, \ad \rbpp(1))  \leq \dim_k\,H^0(I_p, \ad \rbpp(1)^{ss}),$$
it suffices to consider $I_p$ invariants on the semi-simplification $ \ad \rbpp(1)^{ss}$ .
If $\pi_p$ is unramified and ordinary, then it is known \cite[Corollaire 1]{U} (see also \cite[\S 3]{HT}) that $$\rbpp|_{I_p} \sim \begin{pmatrix}
\cycp^{a+b} & * & * & *\\
0 & \cycp^{a} & * & *\\
0 & 0 & \cycp^{b} & *\\
0 & 0 & 0 & 1 \end{pmatrix} $$

Therefore, $$(\ad \rbpp \otimes \cycp)^{ss}|_{I_p} \cong \cycp^{\oplus 3} \oplus \cycp^{b+1} \oplus \cycp^{1-b}\oplus \cycp^{a+1} \oplus \cycp^{1-a} \oplus \cycp^{a-b+1} \oplus \cycp^{1+b-a} \oplus \cycp^{a+b+1} \oplus \cycp^{1-a-b}.$$ The proposition follows by noting that $\cycp$ has order $p-1$.
\end{proof}

\bibliography{GSp4refs}

\end{document}